\documentclass{my_tCON2e}

\def\etal{\mbox{et al.}}

\begin{document}

\markboth{M. R. Sidi Ammi and D. F. M. Torres}{Optimal Control of Nonlocal Thermistor Equations}

\title{Optimal Control of Nonlocal Thermistor Equations}

\author{Moulay Rchid Sidi Ammi$^{a}$
and Delfim F. M. Torres$^{b}$$^{\ast}$\thanks{$^\ast$Corresponding author.
Email: delfim@ua.pt\vspace{6pt}}\\\vspace{6pt}
$^{a}${\em{AMNEA Group, Department of Mathematics,
Faculty of Sciences and Technics,
Moulay Ismail University,
B.P. 509, Errachidia, Morocco}};
$^{b}${\em{Center for Research and Development in Mathematics and Applications,
Department of Mathematics, University of Aveiro,
3810-193 Aveiro, Portugal}}\\\vspace{6pt}
\received{Submitted 21-March-2012; revised 11-June-2012; accepted 13-June-2012;
for publication in \emph{Internat. J. Control}}}

\maketitle


\begin{abstract}
We are concerned with the optimal control problem
of the well known nonlocal thermistor problem, \textrm{i.e.}, in
studying the heat transfer in the resistor device whose electrical
conductivity is strongly dependent on the temperature. Existence
of an optimal control is proved. The optimality system consisting
of the state system coupled with adjoint equations is derived,
together with a characterization of the optimal control.
Uniqueness of solution to the optimality system, and therefore the
uniqueness of the optimal control, is established. The last part
is devoted to numerical simulations.\bigskip



\end{abstract}


\section{Introduction}
\label{sec:1}

Let $\Omega$ be a bounded domain in $\mathbb{R}^{N}$ with a
sufficiently smooth boundary $\partial \Omega$,
and let $Q_{T}=\Omega \times (0, T)$. In this work
we are interested to study an optimal control problem
to the following nonlocal parabolic boundary value problem:
\begin{equation}
\label{eq1}
\begin{gathered}
\frac{\partial u}{\partial t}- \triangle u = \frac{\lambda f(u)}{(
\int_{\Omega} f(u)\, dx)^{2}}\, ,  \mbox{ in } Q_{T}= \Omega
\times (0, T) \,, \\
\frac{\partial u}{\partial\nu} = -\beta u \, ,  \mbox{ on } S_{T}=
\partial \Omega \times (0, T) \,, \\
u(0)= u_{0} \, ,  \mbox{ in } \Omega,
\end{gathered}
\end{equation}
where $\triangle$ is the Laplacian with respect to the spacial
variables, $f$ is supposed to be a smooth function prescribed
below, and $T$ a fixed positive real. Here $\nu$ denotes the
outward unit normal and $\frac{\partial }{\partial \nu}= \nu .
\nabla$ is the normal derivative on $\partial \Omega$. Such
problems arise in many applications, for instance, in studying the
heat transfer in a resistor device whose electrical conductivity
$f$ is strongly dependent on the temperature $u$. The equation
\eqref{eq1} describes the diffusion of the temperature with the
presence of a nonlocal term. Constant $\lambda$ is a dimensionless
parameter, which can be identified with the square of the applied
potential difference at the ends of the conductor. Function
$\beta$ is the positive thermal transfer coefficient, which can
depend only in spatial variables $x$ or time $t$, but for the sake
of generality we take $\beta$ depending in both $x$ and $t$. The
given value $u_{0}$ is the initial condition for temperature. Boundary
conditions are derived from Newton cooling law, sometimes called
Robin conditions or third type boundary conditions.
In the particular case when $\beta = 0$,
we obtain an homogeneous Neumann condition or an adiabatic
condition. Other boundary conditions appear naturally, but for the
sake of simplicity we consider in this paper mixed conditions
only. Recall that under restrictive conditions, \eqref{eq1} is
obtained by reducing the elliptic-parabolic system of partial
differential equations modelling the so-called thermistor:
\begin{equation}
\label{eq12}
\begin{gathered}
u_{t}= \nabla .(k(u)\nabla u) + \sigma (u)|\nabla \varphi |^{2},\\
\nabla ( \sigma (u) \nabla \varphi )= 0,
\end{gathered}
\end{equation}
where $u$ represents the temperature generated
by the electric current flowing through a conductor,
$\varphi$ the electric potential, and $\sigma(u)$ and $k(u)$
the electric and thermal conductivities, respectively.
For more description, we refer to \citep{lac2,tza}.
A throughout discussion about the history of thermistors,
and more detailed accounts of their advantages and applications to
industry, can be found in \citep{mac,psx,kw,MR2805614}. Since the
paper of \citet{jr}, which apparently was the first who
proved the existence of weak solutions to the system \eqref{eq12},
several results were obtained. In \citep{ac}
existence and regularity of weak solutions to the
thermistor problem were established. We remember that existence and uniqueness of
solution to \eqref{eq1} under hypotheses (H1)--(H3) below
(\textrm{cf.} Sec.~\ref{sec:2}) has been established in
\citep{sidi}. For more on existence and uniqueness we refer
to \citep{ajse,MR2724193,MR2805614}.

Optimal control of problems governed by partial differential
equations is a fertile field of research and a source of many
challenging mathematical issues and interesting applications
\citep{Lions71,MR2747291,MR2583281}.
Among essential points in the theory we mention:
(i) existence, regularity, and uniqueness of the optimal
control problem; (ii) necessary optimality conditions,
which consist of the equation under consideration and an adjoint system.
Existence and regularity theory of elliptic and parabolic
equations was developed since \citep{lad}. Optimal
control theory for the system \eqref{eq12} received recently an
important increase of interest. Results for \eqref{eq1} are,
however, scarcer and underdeveloped. To the best of the author's knowledge,
known results on the optimal control of a thermistor problem
reduce to the ones of \citep{LeeShilkin}, where the term source is
taken to be the control. In \citep{gc} the problem of finding the
optimal difference of applied potential to the thermistor problem
\eqref{eq12}, in the sense of minimizing a suitable cost
functional involving the temperature, is studied. Main result of
\citep{gc} gives the optimal system in the simplest case of a
constant electric conductivity. In addition, a theorem of
existence of the optimal solution is given in the general case of
conductivities depending on the temperature. Paper
\citep{MR2405382} investigates a parabolic-elliptic system similar
to \eqref{eq12}, assuming a particular structure of the controls.
In \citep{vsv}, authors considered the optimal control of a two
dimensional steady state thermistor problem. An optimal control
problem of a two dimensional time dependent thermistor system is
considered in \citep{vol}. In \citep{sidel} a similar problem to
\eqref{eq12} is studied, consisting of nonlinear partial
differential equations resulting from the traditional modelling of
oil engineering within the framework of the mechanics of a
continuous medium. The main technique of \citep{sidel} is the
adjoint state and disturbance method to derive the necessary
optimality conditions. Recently, the authors in
\citep{MR2599927} investigated the state-constrained optimal
control of the thermistor problem with the restriction to
two-dimensional domains, while in \citep{MR2805614}
some applications to the thermistor problem, and to certain problems
of filtration of fluids in a porous medium in the presence
of the so-called Soret--Dufour effect, are given.
However, we are not aware of any work or
study about the optimal control of \eqref{eq1}.

It is known that large temperature gradients
may cause a thermistor to crack. Numerical experiments in
\citep{ffh,zw,MR2454231} show that low values of the heat transfer coefficient
$\beta$ results in small temperature variations. On the other hand,
low values of the heat transfer coefficient leads to high
operating temperatures of a thermistor, which is undesirable from
the point of view of applications. This motivates the choice of the heat
transfer coefficient as the control, and to consider the optimal
control problem of minimizing the heat transfer coefficient while
keeping the operating temperature of the thermistor not too high.


\section{Outline of the paper and Hypotheses}
\label{sec:2}

We consider an optimal control problem with
the partial differential equations \eqref{eq1}:

\smallskip

(i) The control $\beta$ belongs
to the set of admissible controls
$$
U_{M}= \left \{ \beta \in L^{\infty}(\Omega \times (0, T))\, , 0 <
m \leq \beta \leq M \right \}.
$$

(ii) The goal is to minimize  a cost functional $J(\beta)$
defined in terms of $u(\beta)$ and $\beta$ as
\begin{equation*}
J(\beta)= \int_{Q_{T}} u dx dt + \int_{S_{T}} \beta^{2} ds dt\, .
\end{equation*}

More precisely, we intend to find $\overline{\beta} \in U_{M}$ such that
\begin{equation}
\label{P}
J(\overline{\beta})= \min_{\beta \in U_{M}} J(\beta).
\end{equation}

In Section~\ref{sec:3}, existence and regularity of the optimal control are
established through a minimizing sequence argument. The energy
estimates, in an appropriate space, and then the class of
weak solutions obtained, allow us to study, in Section~\ref{sec:4},
the optimal control problem and to derive the optimality system.
The obtained necessary optimality conditions consist
of the original state parabolic equation \eqref{eq1}
coupled with the adjoint equations together with a
characterization of the optimal control. In general terms,
the approach used here is close to the method used in \citep{vol}
for investigation of the time dependent thermistor problem.
Since our objective functional depends on $u$, it is differentiated with
respect to the control. We calculate the G\^{a}teaux derivative of
$J$ with respect to $\beta$ in the direction $l$ at the minimizer control
$\beta$. We also need to differentiate $u$ with
respect to the control $\beta$. The difference quotient
$\left(u(\beta + \varepsilon l)-u(\beta)\right)/\varepsilon$
is proved to converge weakly in $H^{1}(\Omega)$ to $\psi$.
As a result, the function $\psi$ verifies a
linear PDE which gives the adjoint system, and an explicit form of
the optimal control is determined. Section~\ref{sec:5} is devoted to the
uniqueness of the solution to the optimality system, and therefore
the uniqueness of the optimal control. Finally, in Section~\ref{sec:6}
we solve the optimality system numerically for a constant case of the
optimization parameter.

In the sequel we shall assume the following assumptions:

\smallskip

(H1) $f: \mathbb{R} \rightarrow \mathbb{R}$ is a positive
Lipshitzian continuous function.

(H2) There exist positive constants $c$ and $\alpha$ such
that $c \leq f(\xi) \leq c |\xi|^{\alpha +1} +c$
for all $\xi \in \mathbb{R}$.

(H3) $u_{0} \in L^{\infty}(\Omega)$.

\bigskip

We say that $u$ is a weak solution to \eqref{eq1} if
\begin{equation}
\label{equa1}
\int_{\Omega} \frac{\partial u}{\partial t} v dx + \int_{\Omega}
\nabla u \nabla v dx + \int_{\partial \Omega} \beta u v ds
= \frac{\lambda }{( \int_{\Omega} f(u)\, dx)^{2}} \int_{\Omega} f(u) v dx \, ,
\end{equation}
for all $v \in H^{1}(\Omega)$.
We use the standard notation for Sobolev spaces. We denote
$\|\cdot\|_{L^{p}(\Omega)}= \|\cdot\|_{p}$ for each $p \in [1, \infty]$.
Along the text constants $c$ are generic, and may change at each occurrence.


\section{Existence of an optimal control}
\label{sec:3}

The proof of existence of an optimal control
(Theorem~\ref{thm:3.1}) is done using
proper estimates (Lemma~~\ref{lem32}).

\begin{theorem}
\label{thm:3.1} Assume that the assumptions (H1)--(H3) hold.
Then, there exists at least an optimal solution $\beta \in L^{\infty}(Q_{T})$
of \eqref{P}. Function $u=u(\beta)$ verifies \eqref{eq1},
in the sense of distributions, with the following regularity:
$u \in C(0, T, L^{2}(\Omega))$,
$\frac{\partial u}{\partial t} \in L^{2}(0, T, H^{-1}(\Omega))$,
$u \in L^{2}(0, T, H^{1}(\Omega))$.
\end{theorem}

\begin{proof}
Let $(\beta_{n})_{n}$ be a minimizing sequence of
$J(\beta)$ in $ U_{M}$. In other words, we have
$$
\lim_{n \rightarrow +\infty} J(\beta_{n})
= \inf_{\beta \in U_{M}} J(\beta) \, .
$$
In order to continue the proof we proceed with
the derivation of a priori estimates:

\begin{lemma}
\label{lem32}
Let $u_{n}=u(\beta_{n})$ be the corresponding solutions to the
weak formulation of \eqref{eq1}. Then
$\|u_{n}\|^{2}_{L^{2}(0, T, H^{1}(\Omega))} + \|u_{n}\|_{2}^{2} \leq c$,
where $c$ is a constant independent of $n$.
\end{lemma}

\begin{proof}
Multiplying the corresponding equations of
\eqref{eq1} by $u_{n}$ and using the fact that
$u_{n} \in L^{\infty}(\Omega)$, we obtain
\begin{equation*}
\begin{split}
\frac{1}{2} \frac{\partial}{\partial t} \|u_{n}\|_{2}^{2}+
\int_{\Omega} |\nabla u_{n}|^{2} dx +  \int_{\partial \Omega} \beta_{n} u_{n}^{2} ds &
\leq c \int_{\Omega}  |f(u_{n})u_{n}| dx \\
& \leq c \int_{\Omega}  \left ( |u_{n}|^{\alpha + 1} + c \right) |u_{n}| dx \\
& \leq c \|u_{n}\|_{1}\\
& \leq c \|u_{n}\|_{2}\\
& \leq c \|u_{n}\|_{H^{1}(\Omega)}.
\end{split}
\end{equation*}
Using the fact that $0 < m \leq \beta_{n}$, we have
\begin{equation}
\label{equation22} \frac{1}{2} \frac{\partial}{\partial t}
\|u_{n}\|_{2}^{2}+ \int_{\Omega} |\nabla u_{n}|^{2} dx + m
\int_{\partial \Omega} u_{n}^{2} ds \leq c \|u_{n}\|_{H^{1}(\Omega)}.
\end{equation}
Now denote
\begin{equation}
\label{eq:ref:ak:1}
\|v\|^{2}_{*}= \int_{\Omega} |\nabla v|^{2} dx +  m \int_{\partial \Omega}
v^{2} ds.
\end{equation}
It is well known that $\|v\|_{*}$ defines a
norm on $H^{1}(\Omega)$ which is equivalent to the
$\|\cdot\|_{H^{1}(\Omega)}$ norm \citep{zeidler}.
Then, there exists a constant $\mu > 0$ such that
\begin{equation}
\label{eq:ref:ak:2}
\mu \|u_{n}\|^{2}_{H^{1}(\Omega)} \leq \|u_{n}\|^{2}_{*} \leq c
\|u_{n}\|^{2}_{H^{1}(\Omega)}.
\end{equation}
It follows from \eqref{equation22}--\eqref{eq:ref:ak:2} that
\begin{equation*}
\frac{1}{2} \frac{\partial}{\partial t} \|u_{n}\|_{2}^{2}+ \mu
\|u_{n}\|^{2}_{H^{1}(\Omega)} \leq  \frac{1}{2}
\frac{\partial}{\partial t} \|u_{n}\|_{2}^{2}
+ \|u_{n}\|^{2}_{*} \leq c \|u_{n}\|_{H^{1}(\Omega)}
\leq \frac{\mu}{2} \|u_{n}\|^{2}_{H^{1}(\Omega)} + c.
\end{equation*}
We obtain
$\|u_{n}(t)\|_{2}^{2}+ \mu \|u_{n}\|^{2}_{L^{2}(0, T, H^{1}(\Omega))} \leq c$
integrating over $(0, T)$.
\end{proof}

We now continue the proof of Theorem~\ref{thm:3.1}.
By Lemma~\ref{lem32} we have, for all $n$, that
$$
u_{n} \in L^{\infty}(0, T, L^{2}(\Omega)) \bigcap L^{2}(0, T,
H^{1}(\Omega)).
$$
Therefore, from \eqref{eq1}, $\frac{\partial u_{n}}{\partial t}$
is bounded in $L^{2}(0, T, H^{-1}(\Omega))$. Using compacity
arguments of Lions \citep{jll} and Aubin's lemma, we have
that $(u_{n})$ is compact in $L^{2}(Q_{T})$. Hence we can extract
from $(u_{n})$ a subsequence, not relabeled, and there exists
$\beta \in U_{M}$ such that
\begin{equation}
\label{eqlimits}
\begin{gathered}
u_{n} \rightarrow u \mbox{ weakly in } L^{2}(0, T,
H^{1}(\Omega)),\\
\frac{\partial u_{n}}{\partial t} \rightarrow \frac{\partial
u_{n}}{\partial t}  \mbox{ weakly in } L^{2}(0, T,
H^{-1}(\Omega)),\\
u_{n} \rightarrow u \mbox{ strongly in } L^{2}(Q_{T}),\\
u_{n} \rightarrow u \mbox{ a.e.  in }L^{2}(Q_{T}),\\
\beta_{n} \rightarrow \beta \mbox{ weakly in } L^{2}(\partial \Omega),\\
\beta_{n} \rightarrow \beta \mbox{ weakly star in } L^{\infty}(\partial \Omega).
\end{gathered}
\end{equation}
Our task consists now to prove that $u=u(\beta)$ is a weak solution of
\eqref{eq1} with control $\beta$. From the weak formulation of $u_{n}$ we have
 $$ \int_{\Omega} \frac{\partial
u_{n}}{\partial t} v dx + \int_{\Omega} \nabla u_{n} \nabla v dx +
\int_{\partial \Omega} \beta_{n} u_{n} v ds = \frac{\lambda }{(
\int_{\Omega} f(u_{n})\, dx )^{2}} \int_{\Omega} f(u_{n}) v dx.
$$
We first show that for any test function $v \in H^{1}(\Omega)$ and
$n \rightarrow \infty$ we have
$$
\int_{\partial \Omega} \beta_{n} u_{n} v ds  \rightarrow
\int_{\partial \Omega} \beta u v ds.
$$
Indeed,
\begin{equation}
\label{eqbound}
\begin{split}
\left |\int_{\partial \Omega} \beta_{n} u_{n} v ds
- \int_{\partial \Omega} \beta u v ds \right | &   \leq  \left |
\int_{\partial \Omega} ( \beta_{n} u_{n} v - \beta_{n} u v ) ds
\right | + \left | \int_{\partial \Omega}
( \beta_{n} u v - \beta u v ) ds  \right |\\
& \leq  M   \int_{\partial \Omega} |u_{n}-
u| |v| ds + \left |  \int_{\partial \Omega} (\beta_{n}- \beta) uv
ds \right | \\
& \leq M \|u_{n}-u\|_{L^{2}(\partial \Omega)}
\|v\|_{L^{2}(\partial \Omega)} |+ \left |  \int_{\partial \Omega}
(\beta_{n}- \beta) uv ds \right | \\
& \leq M \|u_{n}-u\|_{H^{1}(\Omega)}
\|v\|_{L^{2}(\partial \Omega)} |+ \left |  \int_{\partial \Omega}
(\beta_{n}- \beta) uv ds \right |,
\end{split}
\end{equation}
where we used here the trace inequality $\|u\|_{L^{2}(\partial
\Omega)} \leq c \|u\|_{H^{1}(\Omega)}$, which gives that
$u \in H^{1}(\Omega)$ implies $u \in L^{2}(\partial \Omega)$.
It is obvious from limits \eqref{eqlimits} that the right hand
side of the above inequality \eqref{eqbound} goes to $0$ when $n
\rightarrow \infty$.
On the other hand, we have $u_{n} \rightarrow u$ a.e. in
$\Omega \times (0, T)$. Since $f$ is continuous,
$f(u_{n}) \rightarrow f(u) \, a.e. \mbox{ in } L^{2}(\Omega)$.
It follows that
\begin{equation*}
\int_{\Omega} f(u_{n}) dx \rightarrow \int_{\Omega} f(u)
dx,
\end{equation*}
and
\begin{equation*}
\int_{\Omega} f(u_{n})v dx \rightarrow \int_{\Omega} f(u) v dx, \,
\, \forall v \in H^{1}(\Omega).
\end{equation*}
We conclude that $u=u(\beta)$ is a weak solution of \eqref{eq1}.
Using the fact that $J(\beta)$ is weak lower semicontinuous
with respect to the $L^{2}$ norm, it follows that the
infimum is achieved at $\beta$.
\end{proof}


\section{Characterization of the optimal control}
\label{sec:4}

To study the optimal control we derive an optimality system consisting of
equation \eqref{eq1} coupled with an adjoint system. Then, in order to obtain necessary
conditions for the optimality system, we differentiate the cost functional and the
temperature $u$ with respect to the control $\beta$.
Here, besides (H1)--(H3), we further suppose that

\medskip

(H4) $f$ is of class $C^{1}$.

\begin{theorem}
\label{thm31}
Assume hypotheses (H1)--(H4).
Then $\beta \mapsto u(\beta)$ is
differentiable in the sense that as $\varepsilon \rightarrow 0$
$$
\frac{u(\beta + \varepsilon l)-u(\beta)}{\varepsilon}\rightarrow
\psi \mbox{ weakly in } H^{1}(\Omega),
$$
for any $\beta, l \in U_{M}$ such that $(\beta + \varepsilon l)\in
U_{M}$ for small $\varepsilon$. Moreover, $\psi$ verifies
\begin{equation}
\label{sys}
\begin{gathered}
\frac{\partial \psi}{\partial t} - \triangle \psi =
\frac{-2\lambda f(u)}{( \int_{\Omega} f(u)\, dx )^{3}}
\int_{\Omega}f'(u) \psi dx+ \frac{\lambda f'(u) \psi }{(
\int_{\Omega} f(u)\, dx )^{2}}
\quad \mbox{ in } \Omega, \\
\frac{\partial \psi}{\partial \nu}+ \beta \psi + l u = 0 \mbox{ on
} \partial \Omega.
\end{gathered}
\end{equation}
\end{theorem}

The proof of Theorem~\ref{thm31} passes by several steps.


\subsection{A priori estimates and convergence}

Denote $u=u(\beta)$ and $u_{\varepsilon}=u(\beta_{\varepsilon})$,
where $\beta_{\varepsilon}= \beta + \varepsilon l$. Before the
derivation of the optimality system, we need to establish an
$H^{1}$ norm estimate of $\frac{u_{\varepsilon}-u}{\varepsilon}$.

\begin{lemma}
\label{lem33}
We have
$$
\left\|\frac{u_{\varepsilon}-u}{\varepsilon}\right\|_{2}^{2}
+ \left\|\frac{u_{\varepsilon}-u}{\varepsilon}\right\|^{2}_{L^{2}(0, T,H^{1}(\Omega))}
\leq c.
$$
\end{lemma}

\begin{proof}
Subtracting equation \eqref{eq1} from the corresponding equation
of $u_{\varepsilon}$, we have
\begin{equation}
\label{eq4}
\frac{\partial}{\partial t}\left(\frac{u_{\varepsilon}-u}{\varepsilon}\right)
- \triangle\left(\frac{u_{\varepsilon}-u}{\varepsilon}\right)
=\frac{\lambda}{\epsilon} \frac{\left (f(u_{\varepsilon})
-f(u)\right )}{( \int_{\Omega} f(u_{\varepsilon})\, dx)^{2}}
+ \frac{\lambda}{\epsilon} f(u)\left( \frac{1}{( \int_{\Omega}
f(u_{\varepsilon})\, dx )^{2}}
- \frac{1}{( \int_{\Omega} f(u)\, dx)^{2}}\right).
\end{equation}
Multiplying the equation \eqref{eq4} by
$\frac{u_{\varepsilon}-u}{\varepsilon}$, we obtain that
\begin{equation*}
\begin{split}
\frac{1}{2}&\frac{\partial}{\partial t}
\left\|\frac{u_{\varepsilon}-u}{\varepsilon}\right\|_{2}^{2}
+\int_{\Omega} \left|\nabla (\frac{u_{\varepsilon}-u}{\varepsilon})\right|^{2} \, dx
-  \int_{\partial \Omega} \nabla \left(\frac{u_{\varepsilon}-u}{\varepsilon}\right).
\left(\frac{u_{\varepsilon}-u}{\varepsilon}\right) \nu \, ds\\
&= \frac{1}{2}\frac{\partial}{\partial t}
\left\|\frac{u_{\varepsilon}-u}{\varepsilon}\right\|_{2}^{2}
+ \int_{\Omega} \left|\nabla (\frac{u_{\varepsilon}-u}{\varepsilon})\right|^{2} \, dx
- \int_{\partial \Omega}\frac{\partial}{\partial \nu}
\left(\frac{u_{\varepsilon}-u}{\varepsilon}\right)
\left(\frac{u_{\varepsilon}-u}{\varepsilon}\right)  \, ds\\
&= \frac{1}{2}\frac{\partial}{\partial t}
\left\|\frac{u_{\varepsilon}-u}{\varepsilon}\right\|_{2}^{2}
+ \int_{\Omega} \left|\nabla (\frac{u_{\varepsilon}-u}{\varepsilon})\right|^{2} \, dx
+ \int_{\partial \Omega} \beta \left(\frac{u_{\varepsilon}-u}{\varepsilon}\right)^{2} \, ds
+ \int_{\partial \Omega} l u_{\varepsilon}\left(\frac{u_{\varepsilon}-u}{\varepsilon}\right) \, ds\\
&\leq \frac{\lambda}{( \int_{\Omega} f(u)\, dx )^{2}} \left\langle
\frac{f(u_{\varepsilon})- f(u)}{\varepsilon},
\frac{u_{\varepsilon}-u}{\varepsilon} \right\rangle\\
&\qquad + \left\langle \lambda \frac{f(u)}{\varepsilon}
\left( \frac{1}{( \int_{\Omega}
f(u_{\varepsilon})\, dx )^{2}}- \frac{1}{( \int_{\Omega} f(u)\, dx)^{2}}\right),
\frac{u_{\varepsilon}-u}{\varepsilon} \right\rangle \, .
\end{split}
\end{equation*}
Since $0 < m \leq \beta$, we get
\begin{equation*}
\begin{split}
\frac{1}{2}&\frac{\partial}{\partial t}
\left\|\frac{u_{\varepsilon}-u}{\varepsilon}\right\|_{2}^{2}
+ \int_{\Omega} \left|\nabla
\left(\frac{u_{\varepsilon}-u}{\varepsilon}\right)\right|^{2} \, dx
-  \int_{\partial \Omega} \nabla \left(\frac{u_{\varepsilon}-u}{\varepsilon}\right)
\left(\frac{u_{\varepsilon}-u}{\varepsilon}\right) \nu \, ds\\
&\leq \frac{1}{2}\frac{\partial}{\partial t}
\left\|\frac{u_{\varepsilon}-u}{\varepsilon}\right\|_{2}^{2}
+ \int_{\Omega} \left|\nabla
\left(\frac{u_{\varepsilon}-u}{\varepsilon}\right)\right|^{2} \, dx
+ m \int_{\partial \Omega} \left(\frac{u_{\varepsilon}-u}{\varepsilon}\right)^{2} \, ds
+ \int_{\partial \Omega}
l u_{\varepsilon}\left(\frac{u_{\varepsilon}-u}{\varepsilon}\right) \, ds\\
&\leq \frac{\lambda}{( \int_{\Omega} f(u)\, dx )^{2}}
\left\langle \frac{f(u_{\varepsilon})- f(u)}{\varepsilon},
\frac{u_{\varepsilon}-u}{\varepsilon} \right\rangle\\
&\qquad + \left\langle \lambda \frac{f(u)}{\varepsilon}
\left( \frac{1}{( \int_{\Omega} f(u_{\varepsilon})\, dx )^{2}}
- \frac{1}{( \int_{\Omega} f(u)\, dx)^{2}}\right),
\frac{u_{\varepsilon}-u}{\varepsilon} \right\rangle \, .
\end{split}
\end{equation*}
Using the fact that $f$ is Lipschitzian, it follows from the
$L^{\infty}$ boundedness of $u$ and $u_{\varepsilon}$ that
\begin{equation*}
\begin{split}
\frac{1}{2}&\frac{\partial}{\partial t}
\left\|\frac{u_{\varepsilon}-u}{\varepsilon}\right\|_{2}^{2}
+ \int_{\Omega} \left|\nabla (\frac{u_{\varepsilon}-u}{\varepsilon})\right|^{2} \, dx
+ m \int_{\partial \Omega}\left(\frac{u_{\varepsilon}-u}{\varepsilon}\right)^{2} \, ds
+ \int_{\partial \Omega}
l u_{\varepsilon}\left(\frac{u_{\varepsilon}-u}{\varepsilon}\right) \, ds  \\
&\leq  c \int_{\Omega}\left(\frac{u_{\varepsilon}-u}{\varepsilon}\right)^{2} \, dx
+ \frac{\lambda}{\varepsilon} \frac{( \int_{\Omega} f(u)\, dx)^{2}
-\left( \int_{\Omega} f(u_{\varepsilon})\, dx \right)^{2}}{(\int_{\Omega}
f(u_{\varepsilon})\, dx )^{2}
\left( \int_{\Omega} f(u)\, dx \right)^{2}}
\left\langle f(u), \frac{u_{\varepsilon}-u}{\varepsilon} \right\rangle \\
&\leq c  \left\|\frac{u_{\varepsilon}-u}{\varepsilon}\right\|_{2}^{2}
+ \frac{c}{\varepsilon} \left(\int_{\Omega}(f(u)-f(u_{\varepsilon}))\, dx \right)
\left(\int_{\Omega}(f(u)+f(u_{\varepsilon}))\, dx \right)
\left\|\frac{u_{\varepsilon}-u}{\varepsilon}\right\|_{1} \\
&\leq c \left\|\frac{u_{\varepsilon}-u}{\varepsilon}\right\|_{2}^{2}
+ c \left\|\frac{u_{\varepsilon}-u}{\varepsilon}\right\|_{1}
\int_{\Omega}\left(\frac{u_{\varepsilon}-u}{\varepsilon}\right)\, dx
\leq c \left\|\frac{u_{\varepsilon}-u}{\varepsilon}\right\|_{2}^{2}
+ c \left\|\frac{u_{\varepsilon}-u}{\varepsilon}\right\|_{1}^{2} .
\end{split}
\end{equation*}
Since $L^{2}(\Omega) \subseteq L^{1}(\Omega)$, then
\begin{multline*}
\frac{1}{2}\frac{\partial}{\partial t}
\left\|\frac{u_{\varepsilon}-u}{\varepsilon}\right\|_{2}^{2}
+ \int_{\Omega} \left|\nabla
\left(\frac{u_{\varepsilon}-u}{\varepsilon}\right)\right|^{2} \, dx
+ m \int_{\partial \Omega} \left(\frac{u_{\varepsilon}-u}{\varepsilon}\right)^{2} \, ds
+ \int_{\partial \Omega}
l u_{\varepsilon}\left(\frac{u_{\varepsilon}-u}{\varepsilon}\right) \, ds\\
\leq c \left\|\frac{u_{\varepsilon}-u}{\varepsilon}\right\|_{2}^{2}.
\end{multline*}
Using the trace inequality
$\|u_{\varepsilon}\|_{L^{2}(\partial \Omega )}
\leq c \|u_{\varepsilon}\|_{H^{1}(\Omega )}$,  we have
\begin{equation*}
\begin{split}
\frac{1}{2}\frac{\partial}{\partial t} &
\left\|\frac{u_{\varepsilon}-u}{\varepsilon}\right\|_{2}^{2}
+ \int_{\Omega}
\left|\nabla \left(\frac{u_{\varepsilon}-u}{\varepsilon}\right)\right|^{2} \, dx
+ m \int_{\partial \Omega} \left(\frac{u_{\varepsilon}-u}{\varepsilon}\right)^{2} \, ds\\
&\leq \int_{\partial \Omega} |l|
\left|u_{\varepsilon}\right| \left|(\frac{u_{\varepsilon}-u}{\varepsilon})\right| ds +  c
\left\|\frac{u_{\varepsilon}-u}{\varepsilon}\right\|_{2}^{2}\\
& \leq c \left\|u_{\varepsilon}\right\|_{L^{2}(\partial \Omega)}
\left\|\frac{u_{\varepsilon}-u}{\varepsilon}\right\|_{L^{2}(\partial
\Omega)} +  c
\left\|\frac{u_{\varepsilon}-u}{\varepsilon}\right\|_{2}^{2}\\
& \leq c \left\|u_{\varepsilon}\right\|_{H^{1}(\Omega)}
\left\|\frac{u_{\varepsilon}-u}{\varepsilon}\right\|_{H^{1}(\Omega)}
+  c \left\|\frac{u_{\varepsilon}-u}{\varepsilon}\right\|_{2}^{2}.
\end{split}
\end{equation*}
Thus,
\begin{equation*}
\frac{1}{2}\frac{\partial}{\partial t}
\left\|\frac{u_{\varepsilon}-u}{\varepsilon}\right\|_{2}^{2}
+\left\|\frac{u_{\varepsilon}-u}{\varepsilon}\right\|^{2}_{*}
\leq c \left\|u_{\varepsilon}\right\|_{H^{1}(\Omega)}
\left\|\frac{u_{\varepsilon}-u}{\varepsilon}\right\|_{H^{1}(\Omega)}
+  c \left\|\frac{u_{\varepsilon}-u}{\varepsilon}\right\|_{2}^{2}.
\end{equation*}
On the other hand, by the equivalence of $\|\cdot\|_{*}$
and $\|\cdot\|_{H^{1}}$, we have for a positive constant $c$ that
\begin{equation*}
c \left\|\frac{u_{\varepsilon}-u}{\varepsilon}\right\|^{2}_{H^{1}(\Omega)}
\leq \left\|\frac{u_{\varepsilon}-u}{\varepsilon}\right\|^{2}_{*}\, .
\end{equation*}
It follows from Young's inequality that
\begin{equation*}
\begin{split}
\frac{1}{2}\frac{\partial}{\partial t}
\left\|\frac{u_{\varepsilon}-u}{\varepsilon}\right\|_{2}^{2}
+ c \left\|\frac{u_{\varepsilon}-u}{\varepsilon}\right\|^{2}_{H^{1}(\Omega)}
& \leq c \left\|u_{\varepsilon}\right\|_{H^{1}(\Omega)}
\left\|\frac{u_{\varepsilon}-u}{\varepsilon}\right\|_{H^{1}(\Omega)}
+  c \left\|\frac{u_{\varepsilon}-u}{\varepsilon}\right\|_{2}^{2}\\
& \leq c \left\|\frac{u_{\varepsilon}-u}{\varepsilon}\right\|_{H^{1}(\Omega)}
+ c \left\|\frac{u_{\varepsilon}-u}{\varepsilon}\right\|_{2}^{2}\\
& \leq \frac{c}{2} \left\|\frac{u_{\varepsilon}-u}{\varepsilon}\right\|_{H^{1}(\Omega)}^{2}
+ c \left\|\frac{u_{\varepsilon}-u}{\varepsilon}\right\|_{2}^{2} + c .
\end{split}
\end{equation*}
Therefore,
\begin{equation*}
\frac{\partial}{\partial t}
\left\|\frac{u_{\varepsilon}-u}{\varepsilon}\right\|_{2}^{2}
+ c \left\|\frac{u_{\varepsilon}-u}{\varepsilon}\right\|^{2}_{H^{1}(\Omega)}
\leq c.
\end{equation*}
We get the intended result of Lemma~\ref{lem33}
integrating this inequality with respect to time.
\end{proof}

Using the energy estimates of Lemma~\ref{lem33} we have,
up to a subsequence of $\varepsilon \rightarrow 0$, that there exists $\psi$ such that
\begin{equation}
\label{eqconverge}
\begin{gathered}
\frac{u_{\varepsilon}-u}{\varepsilon} \rightarrow \psi \mbox{
weakly in } L^{\infty}(0, T, L^{2}(\Omega)),\\
\frac{u_{\varepsilon}-u}{\varepsilon} \rightarrow \psi \mbox{
weakly in } L^{2}(0, T, H^{1}(\Omega)),\\
\frac{\partial}{\partial t}\left(\frac{u_{\varepsilon}-u}{\varepsilon}\right)
\rightarrow \frac{\partial \psi}{\partial t} \mbox{ weakly in } L^{2}(0, T, H^{-1}\Omega)),\\
\frac{u_{\varepsilon}-u}{\varepsilon} \rightarrow \psi \mbox{
weakly in } L^{\infty}(0, T, L^{2}(\partial \Omega)),\\
\beta_{\varepsilon} \rightarrow \beta \mbox{ weakly in }
L^{2}(\partial \Omega) \mbox{ as } \varepsilon \rightarrow 0,\\
\beta_{\varepsilon} \rightarrow \beta \mbox{ weakly in }
L^{\infty}(\Omega) \mbox{ as } \varepsilon \rightarrow 0.
\end{gathered}
\end{equation}


\subsection{Proof of Theorem~\ref{thm31}}

We are now ready to derive system \eqref{sys}. We have
\begin{equation}
\label{eq5}
\begin{split}
\int_{\Omega} & \frac{\partial}{\partial t}
\left(\frac{u_{\varepsilon}-u}{\varepsilon}\right) v dx
+ \int_{\Omega} \nabla \left(\frac{u_{\varepsilon}-u}{\varepsilon}\right) \nabla v dx
+ \int_{\partial \Omega} \beta \left(\frac{u_{\varepsilon}-u}{\varepsilon}\right) v ds
+ \int_{\partial \Omega} l u_{\varepsilon} v ds \\
&=I + II
\end{split}
\end{equation}
with
$$
I := \frac{\lambda}{\left(\int_{\Omega} f(u_{\varepsilon})\, dx\right)^{2}}
\int_{\Omega}\frac{f(u_{\varepsilon})-f(u)}{\varepsilon} \cdot v dx
$$
and
$$
II := \frac{\lambda}{\varepsilon} \left(
\frac{1}{\left(\int_{\Omega} f(u_{\varepsilon})\, dx\right)^{2}}
- \frac{1}{\left(\int_{\Omega} f(u)\, dx\right)^{2}}\right)
\int_{\Omega} f(u) v \, dx .
$$
We can write $II$ as follows:
\begin{equation*}
\begin{split}
II&= \frac{\lambda}{\varepsilon} \frac{( \int_{\Omega} f(u)\, dx)^{2}
- ( \int_{\Omega} f(u_{\varepsilon})\, dx )^{2}}{(
\int_{\Omega} f(u)\, dx )^{2} ( \int_{\Omega} f(u_{\varepsilon})\,
dx )^{2}} \int_{\Omega}f(u) v dx\\
& = \lambda \int_{\Omega}
\frac{(f(u)-f(u_{\varepsilon}))}{\varepsilon}dx \times
\frac{\int_{\Omega}(f(u)+f(u_{\varepsilon}))dx}{( \int_{\Omega}
f(u_{\varepsilon})\, dx )^{2} ( \int_{\Omega} f(u)\, dx )^{2}}
\int_{\Omega}f(u) v dx \, .
\end{split}
\end{equation*}
One can show, using weak convergence \eqref{eqconverge}, that
\begin{equation*}
II \rightarrow \frac{-2 \lambda \int_{\Omega} f(u)v dx }{(
\int_{\Omega} f(u)\, dx )^{3}} \int_{\Omega} f'(u) \psi dx \mbox{
as } \varepsilon \rightarrow 0.
\end{equation*}
In the same manner we have
\begin{equation*}
I \rightarrow \frac{\lambda}{( \int_{\Omega} f(u)\, dx )^{2}}
\int_{\Omega} f'(u) \psi v dx
 \mbox{ as } \varepsilon \rightarrow 0.
\end{equation*}
Again, from the weak convergence \eqref{eqconverge}, we conclude
that, as $\varepsilon \rightarrow 0$, \eqref{eq5} converges  to
\begin{multline*}
\int_{\Omega} \frac{\partial \psi}{\partial t} v dx +
\int_{\Omega} \nabla \psi \nabla v + \int_{\partial \Omega} \beta
\psi v ds + \int_{\partial \Omega} l u v ds\\
= \frac{-2\lambda \int_{\Omega} f(u) v dx}{( \int_{\Omega} f(u)\,
dx )^{3}} \int_{\Omega}f'(u) \psi dx  + \frac{\lambda}{(
\int_{\Omega} f(u)\, dx )^{2}} \int_{\Omega} f'(u) \psi v dx
\end{multline*}
for every $v \in H^{1}(\Omega)$. In other words,
\begin{multline}
\label{equa7}
\int_{\Omega} \frac{\partial \psi}{\partial t} v dx +
\int_{\Omega} \nabla \psi \nabla v dx + \int_{\partial \Omega}(
\beta \psi + l u) v ds\\
= \frac{-2\lambda \int_{\Omega} f(u) v dx}{( \int_{\Omega} f(u)\,
dx )^{3}} \int_{\Omega}f'(u) \psi dx + \frac{\lambda}{(
\int_{\Omega} f(u)\, dx )^{2}} \int_{\Omega} f'(u) \psi v dx.
\end{multline}
We can rewrite \eqref{equa7} as follows:
\begin{multline*}
\int_{\Omega} \frac{\partial \psi}{\partial t} v dx +
\int_{\Omega} -\triangle \psi  v dx +  \int_{\partial \Omega}(
\frac{\partial \psi}{\partial \nu}+ \beta \psi + l u) v ds\\
= \frac{-2\lambda \int_{\Omega} f(u) v dx}{( \int_{\Omega} f(u)\,
dx )^{3}} \int_{\Omega}f'(u) \psi dx + \frac{\lambda}{(
\int_{\Omega} f(u)\, dx )^{2}} \int_{\Omega} f'(u) \psi v dx.
\end{multline*}
We conclude that $\psi$ satisfies the system
\begin{equation*}
\begin{gathered}
\frac{\partial \psi}{\partial t} - \triangle \psi =
\frac{-2\lambda \int_{\Omega} f'(u) \psi dx}{( \int_{\Omega}
f(u)\, dx )^{3}} f(u) + \frac{\lambda f'(u) \psi}{( \int_{\Omega}
f(u)\, dx )^{2}}\quad \mbox{ in } \Omega \, , \\
\frac{\partial \psi}{\partial \nu}+ \beta \psi + l u = 0
\quad \mbox{ on } \partial \Omega.
\end{gathered}
\end{equation*}
This completes the proof of Theorem~\ref{thm31}.


\subsection{Derivation of the adjoint system}

In order to derive the optimality system and to characterize the
optimal control, we introduce an adjoint function
$\varphi$, defined in $Q_{T}$ and enough smooth,
and the adjoint operator associated with $\psi$.
Multiplying the first equation of \eqref{sys} by $\varphi$ and
integrating in space and time, we have
\begin{multline}
\label{eqadj}
\int_{Q_{T}} \frac{\partial \psi}{\partial t} \cdot \varphi dx dt
+ \int_{Q_{T}} -\triangle \psi \cdot \varphi dx dt\\
=\frac{-2\lambda \int_{\Omega} f'(u) \psi dx}{( \int_{\Omega}
f(u)\, dx )^{3}} \int_{Q_{T}} f(u) \varphi dx dt +
\frac{\int_{Q_{T}}\lambda f'(u) \psi \varphi dx dt}{(
\int_{\Omega} f(u)\, dx )^{2}} \mbox{ in } \Omega \, .
\end{multline}
Integrating by parts \eqref{eqadj} with respect to time, and
imposing the boundary and initial conditions
$$
\frac{\partial \varphi}{\partial \nu} + \beta \varphi = 0
\mbox{ on } \partial \Omega \times (0, T),
\quad \varphi(T) = 0 \, , \quad \varphi(0)= 0\, ,
$$
we obtain
\begin{multline*}
\int_{\Omega} \psi(T) \varphi(T) dx - \int_{\Omega} \psi(0) \varphi(0) dx
+\int_{Q_{T}}- \frac{\partial \varphi}{\partial t}. \psi dx dt
+\int_{Q_{T}} -\triangle   \varphi . \psi dx dt\\
 = \frac{-2\lambda
\int_{Q_{T}} f(u) \varphi dx dt}{( \int_{\Omega} f(u)\, dx )^{3}}
\int_{\Omega} f'(u) \psi dx  + \frac{\lambda \int_{Q_{T}} f'(u)
\varphi \psi dx dt}{( \int_{\Omega} f(u)\, dx )^{2}}.
\end{multline*}
Thus, the function $\varphi$ satisfies the adjoint system given by
\begin{equation}
\label{eq10}
\begin{gathered}
- \frac{\partial \varphi}{\partial t} - \triangle   \varphi
= \frac{-2\lambda \int_{\Omega} f(u) \varphi dx}{( \int_{\Omega}
f(u)\, dx )^{3}} f'(u)+ \frac{\lambda  f'(u) \varphi}{(
\int_{\Omega} f(u)\, dx )^{2}}+1  \mbox{ in } Q_{T},\\
\frac{\partial \varphi}{\partial \nu} + \beta \varphi = 0
\mbox{ on } \partial \Omega \times (0, T),\\
\varphi(0)= 0\, , \quad \varphi(T) = 0 \, ,
\end{gathered}
\end{equation}
where the $1$ appears from differentiation
of the integrand of $J(\beta)$ with respect
to the state $u$.

\begin{theorem}[(Existence of solution to the adjoint system)\ ]
\label{thm}
Given an optimal control $\beta \in U_{M}$ and the corresponding state $u$,
there exists a solution $\varphi$ to the adjoint system \eqref{eq10}.
\end{theorem}

\begin{proof}
Follows by the arguments in \citep{sidel}.
\end{proof}


\subsection{Derivation of the optimality system}

Gathering equation \eqref{eq1} and the adjoint system
\eqref{eq10}, we obtain the following optimality system:
\begin{equation}
\label{eq11}
\begin{gathered}
u_{t}- \triangle u = \frac{\lambda f(u)}{( \int_{\Omega}
f(u)\, dx )^{2}}, \\
- \frac{\partial \varphi}{\partial t} - \triangle   \varphi
= \frac{-2\lambda \int_{\Omega} f(u) \varphi dx}{( \int_{\Omega}
f(u)\, dx )^{3}} f'(u)+ \frac{\lambda  f'(u) \varphi}{(
\int_{\Omega} f(u)\, dx )^{2}}+1  \mbox{ in } Q_{T},\\
\frac{\partial u}{\partial \nu} + \beta u = 0
\mbox{ on } \partial \Omega \times (0, T),\\
\frac{\partial \varphi}{\partial \nu} + \beta \varphi = 0
\mbox{ on } \partial \Omega \times (0, T),\\
\varphi(0)= 0 \, , \quad \varphi(T)  = 0, \\
u(0)=u_0 .
\end{gathered}
\end{equation}

\begin{remark}
The existence of solution to the optimality system \eqref{eq11}
follows from the existence of solution to the state system \eqref{eq1}
and the adjoint system \eqref{eq10}, gathered with the existence of optimal
control.
\end{remark}

We characterize the optimal control with the help
of the arguments of \citep{vol}.

\begin{lemma}
The optimal control $\beta$ is explicitly given by
\begin{equation}
\label{eqcharact}
\beta(x, t)= \min \left ( \max \left( -\frac{u
\varphi}{2}, m \right), M \right).
\end{equation}
\end{lemma}

\begin{proof}
Because the minimum of the cost functional $J$ is achieved at
$\beta$, using \eqref{sys}, the convergence results \eqref{eqconverge},
and the second equation of the system \eqref{eq11},
we have, for a variation $l \in U_{M}$
with $\beta + \epsilon l \in U_{M}$ and $\epsilon > 0$
sufficiently small, that
\begin{equation*}
\begin{split}
0 & \leq \lim_{\epsilon \rightarrow 0} \frac{J(\beta + \epsilon
l)-J(\beta)}{\epsilon} \\
&=\lim_{\epsilon \rightarrow 0} \frac{1}{\epsilon}\left
\{\int_{Q_{T}}\left ( u(\beta + \epsilon l)- u(\beta)\right)dx dt+
\int_{\partial \Omega \times (0, T) } \left( (\beta + \epsilon l)^{2}
-\beta^{2}\right ) \, ds \, dt \right \}\\
& \leq \lim_{\epsilon \rightarrow 0}\int_{Q_{T}} \frac{u(\beta +
\epsilon l)- u(\beta)}{\epsilon} dx dt+ 2\int_{\partial \Omega
\times (0, T) } \beta l \, ds \, dt\\
& \leq \int_{Q_{T}} \psi dx dt + 2\int_{\partial \Omega \times (0,
T) } \beta l \, ds \, dt = \int_{Q_{T}} (\psi, \varphi)(1, 0) dx dt
+ 2\int_{\partial \Omega \times (0,T) } \beta l \, ds \, dt \\
& \leq \int_{\partial \Omega \times (0, T) }\left ( 2\beta l +l u
\varphi \right) \, ds \, dt\\
& \leq \int_{\partial \Omega \times (0, T)}
l \left( 2 \beta + u \varphi \right)\, ds \, dt.
\end{split}
\end{equation*}
Using the arguments and techniques in \citep{vol}
involving choices of the variation function $l$,
we have three cases to distinguish.
(i) Take the variation $\ell$ to have support on the set
$\{x\in \partial \Omega: m<\beta(x,t)<M\}$. The variation
$\ell(x,t)$ can be of any sign, therefore we obtain
$2\beta+u\varphi=0$, whence $\beta=-\frac{u\varphi}{2}$.
(ii) On the set $\{(x,t)\in S_{T}=\partial \Omega
\times(0,T):\beta(x,t)=M\}$, the variation must satisfy
$\ell(x,t)\leq 0$ and therefore we get $2\beta+u\varphi\leq 0$,
implying $M=\beta(x,t)\leq-\frac{u\varphi}{2}$.
(iii)  On the set $\{(x,t)\in\partial
\Omega\times(0,T):\beta(x,t)=m\}$, the variation must satisfy
$\ell(x,t)\geq 0$. This implies $2\beta+u\varphi\geq 0$ and hence
$m=\beta(x)\geq-\frac{u\varphi }{2}$. Combining cases (i), (ii),
and (iii) gives
\begin{equation*}
\beta =
\begin{cases}
- \frac{u \varphi}{2}
& \mbox{ if } m < - \frac{u \varphi}{2} < M ,\\
M,
& \mbox{ if  } - \frac{u \varphi}{2} \geq  M,\\
m & \mbox{ if } - \frac{u \varphi}{2} \leq m .
\end{cases}
\end{equation*}
This can be written compactly as \eqref{eqcharact}.
\end{proof}


\subsection{Particular case: a constant heat transfer coefficient}

Let us consider now the case when the heat transfer coefficient
$\beta$ is a constant, \textrm{i.e.}, when $\beta$ is independent
of $x$ and $t$, and
\begin{equation}
\label{eq:J:pc}
J= \int_{\Omega} u dx + \beta^{2}.
\end{equation}
We need to adjust the parameter $\beta \in U_{M}$ in such way that
the new form of the functional \eqref{eq:J:pc} is minimized. Then,
all the theory of existence of optimal control and derivation of
the optimality system, that one needs to put into the proofs of
the previous sections  carries over to this case and are simpler.
As for the characterization of optimal control, we have:

\begin{lemma}
The optimal parameter characterization related to \eqref{eq:J:pc}
is
\begin{equation}
\label{eq23}
\beta = \min \left ( \max \left ( m, -\frac{1}{2} \int_{\partial
\Omega} u \varphi ds \right ), M \right ).
\end{equation}
\end{lemma}

\begin{proof}
For the characterization of the optimal control we take
into account the new expression of the cost functional \eqref{eq:J:pc}:
\begin{equation*}
\begin{split}
0 & \leq \lim_{\epsilon \rightarrow 0} \frac{J(\beta + \epsilon
l)-J(\beta)}{\epsilon} \\
&=\lim_{\epsilon \rightarrow 0} \frac{1}{\epsilon}\left \{
\int_{\Omega} u(\beta +\epsilon l)dx + (\beta +\epsilon l)^{2}-
\int_{\Omega} u(\beta)dx - \beta^{2}
 \right \}\\
&=\lim_{\epsilon \rightarrow 0} \left\{  \int_{\Omega} \frac{u(\beta
+\epsilon l) - u(\beta)}{\epsilon} dx +
\frac{(\beta +\epsilon l)^{2} - \beta^{2}}{\epsilon} \right\} \\
 & = \int_{ \Omega} \psi dx + 2 \beta l \\
& = \int_{ \Omega} (\psi,   \varphi)(1, 0) dx + 2 \beta l .
\end{split}
\end{equation*}
Multiplying the optimality system \eqref{eq11} with the test
function $\psi$, integrating by parts, and using \eqref{sys}, we find that
\begin{equation*}
\begin{split}
0 & \leq \lim_{\epsilon \rightarrow 0} \frac{J(\beta + \epsilon
l)-J(\beta)}{\epsilon}  = l \int_{\partial \Omega} u \varphi ds  +
2 \beta l .
\end{split}
\end{equation*}
Therefore,
$$
l\left(\int_{\partial \Omega} u \varphi ds  + 2 \beta\right) \geq 0.
$$
Repeating all the steps as those yielding to
\eqref{eqcharact}, we obtain that the optimal parameter $\beta$ is
characterized by \eqref{eq23}.
\end{proof}


\section{Uniqueness of the optimal control}
\label{sec:5}

The uniqueness of the optimal control is mainly based on the
$L^{\infty}$ boundedness of $u$ and $\varphi$. These are quite
realistic assumptions since physical quantities are always bounded.
It has been shown in \citep{ajse} that
$u \in L^{\infty}(\Omega)$.
It remains to establish that $\varphi$ is also
essentially bounded.

\begin{lemma}
Under hypotheses (H1)--(H4) one has
$\varphi \in L^{\infty}(\Omega)$.
\end{lemma}

\begin{proof}
Multiplying the second equation of \eqref{eq11}, governed by
$\varphi$, by $- \varphi^{k+1}$ for some enough big integer
$k > 2$, we have by the $L^{\infty}$ estimate
of $u$ and Young's inequality that
\begin{equation*}
\begin{split}
\frac{1}{k+2} \frac{\partial}{\partial t} \|\varphi\|_{k+2}^{k+2}
+ \int_{\Omega} \nabla \varphi \nabla(|\varphi|^{k+1}) dx
& \leq c \|\varphi\|_{L^{1}(\Omega)}\|\varphi\|_{L^{k+1}(\Omega)}^{k+1}
+ c\|\varphi\|_{L^{k+2}(\Omega)}^{k+2}
+ \|\varphi\|_{L^{k+1}(\Omega)}^{k+1} \\
& \leq c \|\varphi\|_{L^{k+1}(\Omega)}^{k+2}
+ c\|\varphi\|_{L^{k+2}(\Omega)}^{k+2}+ c\|\varphi\|_{L^{k+2}(\Omega)}^{k+1}\\
& \leq  c\|\varphi\|_{L^{k+2}(\Omega)}^{k+2}
+ c \|\varphi\|_{L^{k+2}(\Omega)}^{k+1}.
\end{split}
\end{equation*}
Then,
\begin{equation*}
\frac{1}{k+2} \frac{\partial}{\partial t} \|\varphi\|_{k+2}^{k+2}
+ (k+1)\int_{\Omega} |\varphi|^{k-2} |\nabla \varphi|^{2}dx
\leq  c\|\varphi\|_{L^{k+2}(\Omega)}^{k+2}+ c \|\varphi\|_{L^{k+2}(\Omega)}^{k+1}.
\end{equation*}
Taking into account that the second term of the left hand side
is positive, we have
\begin{equation*}
\frac{1}{k+2} \frac{\partial}{\partial t} \|\varphi\|_{k+2}^{k+2}
\leq  c\|\varphi\|_{L^{k+2}(\Omega)}^{k+2}+ c \|\varphi\|_{L^{k+2}(\Omega)}^{k+1}.
\end{equation*}
Setting $y_{k}= \|\varphi\|_{L^{k+2}(\Omega)}$, it follows that
$$
y_{k}^{k+1} \frac{\partial y_{k}}{\partial t} \leq c y_{k}^{k+2}+
c y_{k}^{k+1}.
$$
In other words,
$$
\frac{\partial y_{k}}{\partial t} \leq c y_{k} + c.
$$
By the Gronwall Lemma we have
$y_{k} =  \|\varphi\|_{L^{k+2}(\Omega)}  \leq  c$,
where $c$ are constants independent of $k$.
Letting $k \rightarrow \infty$, we have
$\|\varphi\|_{L^{\infty}(\Omega)} \leq c$.
\end{proof}

\begin{theorem}
\label{thm41}
If the hypotheses (H1)--(H4) hold, then the solution of the optimality
system \eqref{eq11} is unique and, therefore, the optimal control $\beta$ is unique.
\end{theorem}

\begin{proof}
Let $u_{1}$, $\varphi_{1}$ and $u_{2}$, $\varphi_{2}$ be two
solutions to the optimality system \eqref{eq11} and $\beta_{1}$,
$\beta_{2}$ be two optimal controls. Denote $w= u_{1}- u_{2}$
and $\varphi= \varphi_{2}- \varphi_{1}$. Upon subtracting
and estimating the difference between the equations governed
by $u_{1}$ and $u_{2}$, we have
\begin{equation}
\label{eq12:rf}
\frac{\partial}{\partial t}w - \triangle w= \frac{\lambda
f(u_{1})}{( \int_{\Omega} f(u_{1})\, dx )^{2}}- \frac{\lambda
f(u_{2})}{( \int_{\Omega} f(u_{2})\, dx)^{2}}
= g(u_{1}, u_{2})w,
\end{equation}
where
$$
g(u_{1}, u_{2})= \left (\frac{\lambda f(u_{1})}{(
\int_{\Omega} f(u_{1})\, dx )^{2}}
- \frac{\lambda f(u_{2})}{\int_{\Omega} f(u_{2})\, dx} \right)
/ \left(u_{2}-u_{1}\right).
$$
By using hypotheses (H1)--(H4) and the $L^{\infty}$
estimate of $u_{i}$, $i=1, 2$, we have
$g(u_{1}, u_{2}) \in L^{\infty}(\Omega)$.
Multiplying \eqref{eq12:rf} by $w$ yields
\begin{equation*}
\frac{1}{2}\frac{\partial \|w\|_{2}^{2}}{\partial t} +
\int_{\Omega}  |\nabla w|^{2} dx + \int_{\partial \Omega}
\beta_{2} |w|^{2} ds + \int_{\partial \Omega}
(\beta_{2}-\beta_{1})u_{1} w ds \leq c \|w\|_{2}^{2}.
\end{equation*}
Since $m \leq \beta_{2}$, we have
\begin{equation*}
\frac{1}{2}\frac{\partial}{\partial t} \|w\|_{2}^{2}+
\int_{\Omega}  |\nabla w|^{2} dx + m \int_{\partial \Omega}
|w|^{2} ds \leq c \|w\|_{2}^{2} + \int_{\partial \Omega}
|\beta_{2}-\beta_{1}||u_{1}| |w| ds.
\end{equation*}
It follows from
$\beta_{i}= \min \left ( \max (-\frac{u_{i} \varphi_{i}}{2}, m), M\right)$,
$i= 1, 2$, that
$$
|\beta_{2}- \beta_{1}| \leq \frac{1}{2} |u_{2}\varphi_{2}-
u_{1}\varphi_{1}| \leq \frac{1}{2}( |u_{2}\varphi| + |w
\varphi_{1}|).
$$
Then we get
\begin{equation*}
\begin{split}
\frac{1}{2}&\frac{\partial}{\partial t} \|w\|_{2}^{2}+
\int_{\Omega}  |\nabla w|^{2} dx + m \int_{\partial \Omega}
|w|^{2} ds \\
& \leq c \|w\|_{2}^{2} + \frac{1}{2} \int_{\partial
\Omega} |u_{2}\varphi_{2}- u_{1}\varphi_{1}|\, |u_{1}|\, |w| ds \\
& \leq c \|w\|_{2}^{2} + \frac{1}{2} \int_{\partial
\Omega}  ( |u_{2}\varphi| + |w
\varphi_{1}|)  \, |u_{1}|\, |w| ds \\
& \leq c \|w\|_{2}^{2} + \frac{1}{2} \|u_{1}\|_{\infty}
\|u_{2}\|_{\infty} \|\varphi\|_{L^{2}(\partial \Omega)}
\|w\|_{L^{2}(\partial \Omega)}
+ \frac{1}{2} \|u_{1}\|_{\infty} \|\varphi_{1}\|_{\infty}
\|w\|^{2}_{L^{2}(\partial \Omega)}
\end{split}
\end{equation*}
and, using Young's inequality,
\begin{equation}
\label{eq:d}
\frac{1}{2}\frac{\partial}{\partial t} \|w\|_{2}^{2}+
\int_{\Omega}  |\nabla w|^{2} dx +\left(
m-\frac{1}{2}\|u_{1}\|_{\infty} \|\varphi_{1}\|_{\infty}- c \right
) \|w\|^{2}_{L^{2}(\partial \Omega)}   \leq c \|w\|_{2}^{2} + c
\|\varphi\|_{L^{2}(\partial \Omega)}^{2}.
\end{equation}
On the other hand, using the adjoint system, we have
\begin{equation}
\label{eq13}
- \frac{\partial \varphi_{2}}{\partial t} - \triangle \varphi_{2}=
h_{2}(u_{1}, u_{2}, \varphi_{1},
\varphi_{2})(\varphi_{2}-\varphi_{1})
\end{equation}
and
\begin{equation}
\label{eq14}
- \frac{\partial \varphi_{1}}{\partial t} - \triangle \varphi_{1}=
h_{1}(u_{1}, u_{2}, \varphi_{1},
\varphi_{2})(\varphi_{2}-\varphi_{1})\, ,
\end{equation}
where
$$
h_{1}(u_{1}, u_{2}, \varphi_{1}, \varphi_{2})
 = \left ( -\frac{2\lambda \int_{\Omega} f(u_{1}) \varphi_{1} dx }{( \int_{\Omega}
f(u_{1})\, dx )^{3}} f'(u_{1})+ \frac{\lambda  f'(u_{1})
\varphi_{1}}{( \int_{\Omega} f(u_{1})\, dx )^{2}}+1 \right)
/ (\varphi_{2}- \varphi_{1})
$$
and
$$
 h_{2}(u_{1}, u_{2}, \varphi_{1}, \varphi_{2}) =
 \left ( -\frac{2\lambda \int_{\Omega} f(u_{2}) \varphi_{2} dx}{( \int_{\Omega}
f(u_{2})\, dx )^{3}} f'(u_{2})+ \frac{\lambda  f'(u_{2})
\varphi_{2}}{( \int_{\Omega} f(u_{2})\, dx )^{2}}+1 \right)
/(\varphi_{2}- \varphi_{1}) \, .
$$
Note that $h_{1}(u_{1}, u_{2}, \varphi_{1}, \varphi_{2}),
h_{2}(u_{1}, u_{2}, \varphi_{1}, \varphi_{2}) \in
L^{\infty}(\Omega)$. Subtracting \eqref{eq13} from \eqref{eq14},
we get
\begin{equation*}
\frac{\partial (\varphi_{2}-\varphi_{1}) }{\partial t}+ \triangle
(\varphi_{2}-\varphi_{1})= \left(h_{1}(u_{1}, u_{2}, \varphi_{1},
\varphi_{2})-h_{2}(u_{1}, u_{2}, \varphi_{1}, \varphi_{2})
\right)(\varphi_{2}-\varphi_{1})\, .
\end{equation*}
Multiplying the above equation by
$\varphi= \varphi_{2}-\varphi_{1}$,
using hypotheses, and the $L^{\infty}$ estimates
of $u_{1}, u_{2}, \varphi_{1}, \varphi_{2}, h_{1}, h_{2}$, we get
$$
\frac{1}{2}\frac{\partial}{\partial t}
\|\varphi_{2}-\varphi_{1}\|_{2}^{2}- \left \{ \int_{\Omega}
|\nabla (\varphi_{2}- \varphi_{1})|^{2}dx - \int_{\partial \Omega}
\frac{\partial}{\partial \nu} (\varphi_{2}
- \varphi_{1})(\varphi_{2}- \varphi_{1}) ds \right \} \leq c
\|\varphi_{2}-\varphi_{1}\|_{2}^{2}.
$$
Then
\begin{equation*}
\begin{split}
\frac{1}{2}\frac{\partial}{\partial t} \|\varphi\|_{2}^{2} -
\int_{\Omega} |\nabla \varphi|^{2}dx + \int_{\partial \Omega}
\left( \beta_{1} \varphi_{1}- \beta_{2}
\varphi_{2}\right)(\varphi_{2}-\varphi_{1}) ds & \leq c
\|\varphi\|_{2}^{2}
\end{split}
\end{equation*}
and it follows that
\begin{equation*}
\frac{1}{2}\frac{\partial}{\partial t} \|\varphi\|_{2}^{2} -
\int_{\Omega} |\nabla \varphi|^{2}dx + \int_{\partial \Omega}
\beta_{2}|\varphi|^{2} ds + \int_{\partial \Omega}
(\beta_{2}-\beta_{1})\varphi_{1}\varphi ds \leq
 c\|\varphi\|_{2}^{2}.
\end{equation*}
We have $\beta_{i}= \max \left(\min (m,
-\frac{u_{i}\varphi_{i}}{2}), M \right)$. Therefore,
$$
|\beta_{2}-\beta_{1}| \leq |u_{2} \varphi_{2}- u_{1} \varphi_{1}|
\leq \frac{1}{2} (|\varphi_{2} w| + |u_{1} \varphi|).
$$
Using the fact that $m \leq \beta_{2}$, we have
\begin{equation*}
\begin{split}
\frac{1}{2}&\frac{\partial}{\partial t} \|\varphi\|_{2}^{2} +m
\int_{\partial \Omega} |\varphi|^{2} ds \\
& \leq \int_{\Omega} |\nabla \varphi|^{2}dx + c\|\varphi\|_{2}^{2}
+ \frac{1}{2} \int_{\partial \Omega} |u_{2} \varphi_{2}
- u_{1} \varphi_{1}| \, |\varphi_{1}\varphi| ds\\
& \leq \int_{\Omega} |\nabla \varphi|^{2}dx + c\|\varphi\|_{2}^{2} +
\frac{1}{2} \int_{\partial \Omega} (|\varphi_{2} w|
+ |u_{1} \varphi|)\, |\varphi_{1}\varphi| ds \\
& \leq \int_{\Omega} |\nabla \varphi|^{2}dx + c\|\varphi\|_{2}^{2}
+\frac{1}{2}\|\varphi_{2}\|_{\infty} \|\varphi_{1}\|_{\infty}
\|w\|_{L^{2}(\partial \Omega)}\|\varphi\|_{L^{2}(\partial
\Omega)}+ \frac{1}{2}\|\varphi_{1}\|_{\infty} \|u_{1}\|_{\infty}
\|\varphi\|_{L^{2}(\partial \Omega)}^{2}.
\end{split}
\end{equation*}
Using again Young's inequality, we get
\begin{equation} \label{eqs22}
\frac{1}{2}\frac{\partial}{\partial t} \|\varphi\|_{2}^{2} +\left
\{ m- \frac{1}{2}\|\varphi_{1}\|_{\infty} \|u_{1}\|_{\infty} -c
\right \}\|\varphi\|_{L^{2}(\partial \Omega)}^{2}
\leq \int_{\Omega} |\nabla \varphi|^{2}dx + c\|\varphi\|_{2}^{2} +
c\|w\|_{L^{2}(\partial \Omega)}^{2}
\end{equation}
and, from Poincar\'e's inequality and the fact that the operator trace
from $H^{1}(\Omega)$ to the boundary space $L^{2}(\partial \Omega)$ is
linear and compact, we have from \eqref{eq:d} and \eqref{eqs22} that
\begin{multline*}
\frac{1}{2} \frac{\partial}{\partial t} \left ( \|w\|_{2}^{2}
+  \|\varphi\|_{2}^{2}\right ) + \left(
m-\frac{1}{2}\|u_{1}\|_{\infty} \|\varphi_{1}\|_{\infty}- c \right)
\|w\|^{2}_{L^{2}(\partial \Omega)}\\
+ \left
\{ m- \frac{1}{2}\|\varphi_{1}\|_{\infty} \|u_{1}\|_{\infty}
-c \right \}\|\varphi\|_{L^{2}(\partial \Omega)}^{2}
\leq \int_\Omega |\nabla \varphi|^2 dx + c \|w\|_{2}^{2}
+ c\|\varphi\|_{2}^{2} + c\|w\|^{2}_{L^{2}(\partial \Omega)}
+ c \|\varphi\|_{L^{2}(\partial \Omega)}^{2}.
\end{multline*}
Then,
\begin{multline*}
\frac{1}{2} \frac{\partial}{\partial t} \left ( \|w\|_{2}^{2}
+  \|\varphi\|_{2}^{2}\right )
+ \left(
m-\frac{1}{2}\|u_{1}\|_{\infty} \|\varphi_{1}\|_{\infty}- c \right)
\|w\|^{2}_{L^{2}(\partial \Omega)}\\
+ \left
\{ m- \frac{1}{2}\|\varphi_{1}\|_{\infty} \|u_{1}\|_{\infty}
-c \right \}\|\varphi\|_{L^{2}(\partial \Omega)}^{2}  \leq c \|w\|_{2}^{2}
+ c\|\varphi\|_{2}^{2}
\end{multline*}
and for $m$ sufficiently large one has
\begin{equation}
\label{equation2}
\frac{\partial}{\partial t} \left ( \|w\|_{2}^{2}
+  \|\varphi\|_{2}^{2}\right ) \leq c ( \|w\|_{2}^{2} +  \|\varphi\|_{2}^{2}).
\end{equation}
Gronwall's inequality leads to
$\|w\|_{2}^{2} + \|\varphi\|_{2}^{2} \leq 0$.
Then $u_{1}=u_{2}$ and $\varphi_{2}= \varphi_{1}$, which gives the
uniqueness of solutions to the optimality system and therefore the
uniqueness of the optimal control, since we have the existence of
an optimal control and corresponding state and adjoint, which
satisfy the optimality system. This completes the proof of
Theorem~\ref{thm41}.
\end{proof}

\begin{remark}
The uniqueness of the optimal control can be obtained from
$$
|\beta_{2}- \beta_{1}| \leq \frac{1}{2} |u_{2}\varphi_{2}-
u_{1}\varphi_{1}| \leq \frac{1}{2}( |u_{2}\varphi| + |w
\varphi_{1}|),
$$
since $\varphi = w=0$.
\end{remark}


\section{Numerical Example}
\label{sec:6}

We now give a numerical example for a particular problem.
We use a finite element approach based on the
Galerkin method to obtain approximate steady state solutions of
the optimality system in the one-dimensional case. The formulation
of the finite element method is based on a
variational formulation of the continuous optimality system. The
optimality system is discretized by finite differences. We then
obtain the following one-dimensional nonlocal thermistor problem:
\begin{equation*}
\frac{\partial u}{\partial t} - \Delta u =\frac{\lambda f(u)}{(
\int_{\Omega} f(u)\, dx)^{2}}, \, \quad 0< x < 1, \, \quad  t>0,
\end{equation*}
subject to the boundary and initial conditions
\begin{equation*}
\frac{\partial u}{\partial x} = -\beta u \, \quad \mbox { on }
\partial \Omega \times (0, T),
\end{equation*}
\begin{equation*}
u(x, 0)= 0, \, \quad 0\leq x \leq 1.
\end{equation*}

We divide the interval $\Omega = [0, 1]$ into $N$ equal finite
elements $0=x_{0} < x_{1} < \ldots < x_{N}=1$.
Let $(x_{j},x_{j+1})$ be a partition of $\Omega$ and
$x_{j+1}-x_{j}=h=\frac{1}{N}$ the step length.
By $S$ we denote a basis of  the
usual pyramid functions:
\begin{equation*}
v_{j}=
\begin{cases}
\frac{1}{h}x+(1-j)
& \mbox{ on } [x_{j-1}, x_{j}],\\
- \frac{1}{h}x+(1+j)
& \mbox{ on } [x_{j}, x_{j+1}],\\
0 & \mbox{ otherwise}.
\end{cases}
\end{equation*}
First, we write the problem in weak or
variational form. We multiply the parabolic equation by $v_{j}$
(for $j$ fixed), integrate over $(0, 1)$, and apply Green's
formula on the left-hand side, to obtain
\begin{equation*}
\int_{\Omega} \frac{\partial u}{\partial t} v_{j} \, dx +
\int_{\Omega}  \nabla u \nabla v_{j} \, dx - \int_{\partial
\Omega} \frac{\partial u}{\partial \nu} v_{j} \, ds =
\frac{\lambda \int_{\Omega}f(u)v_{j}dx}{( \int_{\Omega} f(u)\,
dx)^{2}} \, .
\end{equation*}
Using the boundary condition we get
\begin{equation}
\label{eq7}
\int_{\Omega} \frac{\partial u}{\partial t} v_{j} \,
dx + \int_{\Omega}  \nabla u \nabla v_{j} \, dx + \int_{\partial
\Omega}\beta  u v_{j} \, ds = \frac{\lambda
\int_{\Omega}f(u)v_{j}dx}{( \int_{\Omega} f(u)\, dx)^{2}}.
\end{equation}

We now turn our attention to the solution of system \eqref{eq7} by
discretization with respect to the time variable. We introduce a
time step $\tau$ and time levels $t_{n}= n\tau$, where $n$ is a
nonnegative integer, and denote by $u^{n}$ the approximation of
$u(t_{n})$ to be determined. We use the backward Euler--Galerkin
method, which is defined by replacing the time derivative in
\eqref{eq7} by a backward difference $\frac{u^{n+1}-u^{n}}{\tau}$.
So the approximations $u^{n+1}$ admit a unique representation,
$$
u^{n+1}= \sum_{i=-1}^{N}\alpha_{i}^{n+1} v_{i} \, ,
$$
where $\alpha_{i}^{n+1}$ are unknown real coefficients to be
determined. Thus,
\begin{equation*}
\int_{\Omega} \frac{u^{n+1}-u^{n}}{\tau} v_{j} \, dx +
\int_{\Omega}
 \nabla u^{n+1} \nabla v_{j} \, dx +\int_{\partial \Omega}
\beta u^{n+1} v_{j} \, ds =  \frac{\lambda \int_{\Omega} f(u^{n})
 v_{j} dx }{( \int_{\Omega} f(u^{n})\, dx)^{2}} .
\end{equation*}
The scheme may be stated in terms of the functions $v_{i}$:
find the coefficients $\alpha_{i}^{n+1}$
in $u^{n+1}= \sum_{i=-1}^{N}\alpha_{i}^{n+1} v_{i}$ such that
\begin{multline}
\label{eq14b}
\sum_{i=-1}^{N} \alpha_{i}^{n+1} \int_{\Omega}
v_{i}v_{j}\, dx + \tau \sum_{i=-1}^{N} \alpha_{i}^{n+1}
\int_{\Omega} \nabla v_{i}
\nabla v_{j}\, dx +\tau \int_{\partial \Omega}
\beta u^{n+1}  v_{j} \, ds \\
= \sum_{i=-1}^{N} \alpha_{i}^{n} \int_{\Omega} v_{i}v_{j}\, dx
+ \tau \frac{\lambda \int_{\Omega} f(u^{n})
v_{j} dx }{\left( \int_{\Omega} f(u^{n})\, dx\right)^{2}}.
\end{multline}
In matrix notation, this may be expressed as
$\left (A+ \tau B \right) \alpha^{n+1}= g^{n}= g(n\tau)$,
where
$$
A= (a_{ij}) \mbox{ with element }a_{ij}
= \int_{\Omega} v_{i}v_{j}\, dx \, ,
$$
$$
B= (b_{ij}) \mbox{ with  } b_{ij}= \int_{\Omega}  \nabla v_{i}
\nabla v_{j}\, dx \, ,
$$
and
$\alpha^{n+1}$ is the vector of unknowns
$(\alpha_{i}^{n+1})_{i=-1}^{N}$.
Since the matrix $A$ and $B$ are Gram matrices, in particular they
are positive definite and invertible. Thus, the above system of
ordinary differential equations has obviously a unique solution.
We solve the system \eqref{eq14b} for each time level. Estimating
each term of \eqref{eq14b} separately, we have:
\begin{equation*}
\begin{split}
\sum_{i=-1}^{N} & \alpha_{i}^{n+1} \int_{\Omega} v_{i}v_{j}\, dx\\
&=\sum_{i=-1}^{N} \alpha_{i}^{n+1} \int_{0}^{1}
v_{i}v_{j}\, dx \\
&= \alpha_{j-1}^{n+1} \int_{x_{j-1}}^{x_{j}} v_{j-1}v_{j}\, dx
+\alpha_{j}^{n+1} \int_{x_{j-1}}^{x_{j+1}} v_{j}^{2}\, dx
+ \alpha_{j+1}^{n+1} \int_{x_{j}}^{x_{j+1}} v_{j}v_{j+1}\, dx\\
&= \alpha_{j-1}^{n+1} \int_{x_{j-1}}^{x_{j}} v_{j-1}v_{j}\, dx
+\alpha_{j}^{n+1}\left( \int_{x_{j-1}}^{x_{j}} v_{j}^{2}\, dx
+\int_{x_{j}}^{x_{j+1}} v_{j}^{2}\, dx  \right)
+ \alpha_{j+1}^{n+1} \int_{x_{j}}^{x_{j+1}} v_{j}v_{j+1}\, dx.
\end{split}
\end{equation*}
Using the expression of $v_{j-1}, v_{j}$ and $v_{j+1}$, we obtain
\begin{equation}
\label{eq15}
\sum_{i=-1}^{N} \alpha_{i}^{n+1} \int_{\Omega} v_{i}v_{j}\, dx=
\frac{h}{6}\alpha_{j-1}^{n+1}+\frac{2h}{3}\alpha_{j}^{n+1}+
\frac{h}{6}\alpha_{j+1}^{n+1}.
\end{equation}
Similarly, we have
\begin{equation}
\begin{split}
\sum_{i=-1}^{N} & \alpha_{i}^{n+1} \int_{\Omega} \nabla v_{i}
\nabla v_{j}\, dx\\
&= \sum_{i=-1}^{N} \alpha_{i}^{n+1} \int_{\Omega}
\frac{\partial v_{i}}{\partial x}
\frac{\partial v_{j}}{\partial x} \, dx \\
& = \alpha_{j-1}^{n+1} \int_{x_{j-1}}^{x_{j}}
\frac{\partial v_{j-1}}{\partial x}
\frac{\partial v_{j}}{\partial x} \, dx
+ \alpha_{j}^{n+1} \int_{x_{j-1}}^{x_{j+1}}\left(
\frac{\partial v_{j}}{\partial x}\right)^{2} \, dx
+ \alpha_{j+1}^{n+1} \int_{x_{j}}^{x_{j+1}} \frac{\partial v_{j}}{\partial x}
\frac{\partial v_{j+1}}{\partial x} \, dx,\\
& = - \frac{\alpha_{j-1}^{n+1}}{h^{2}} \int_{x_{j-1}}^{x_{j}} dx
+ \frac{\alpha_{j}^{n+1}}{h^{2}} \int_{x_{j-1}}^{x_{j+1}} dx
- \frac{\alpha_{j+1}^{n+1}}{h^{2}} \int_{x_{j}}^{x_{j+1}} dx \\
& = - \frac{1}{h}\alpha_{j-1}^{n+1} + \frac{2}{h}
\alpha_{j}^{n+1}- \frac{1}{h}\alpha_{j+1}^{n+1}.
\end{split}
\end{equation}
On the other hand,
\begin{equation}
\int_{\Omega} u^{n}v_{j}= \sum_{i=-1}^{N}
\alpha_{i}^{n}\int_{\Omega} v_{i} v_{j} \, dx=
\frac{h}{6}\alpha_{j-1}^{n}+ \frac{2h}{3} \alpha_{j}^{n}+
\frac{h}{6}\alpha_{j+1}^{n}
\end{equation}
and
\begin{equation}
\begin{split}
\beta \int_{\partial \Omega = \{ 0, 1\}} u^{n+1} v_{j}
&\simeq
\frac{1}{2} \left( \beta u^{n+1}(1) v_{j}(1) + \beta u^{n+1}(0) v_{j}(0) \right)\\
& = \frac{1}{2} \left( \beta \alpha_{N}^{n+1} v_{j}(1)
+ \beta \alpha_{0}^{n+1} v_{j}(0) \right) \\
&= \begin{cases}
\frac{1}{2} \beta \alpha_{0}^{n+1}
& \mbox{ if } j= 0,\\
0 &  \mbox{ if } j=1 \ldots N-2,\\
0 & \mbox{ if } j=N-1.
\end{cases}
\end{split}
\end{equation}
Furthermore,
\begin{equation}
\label{eq35}
\frac{\lambda \int_{\Omega} f(u^{n}) v_{j} dx}{\left(
\int_{\Omega} f(u^{n})\, dx\right)^{2}} \simeq
\begin{cases}
\frac{2\lambda f(\alpha_{0}^{n})}{(
f(\alpha_{0}^{n}+ f(\alpha_{N}^{n}))^{2}}
& \mbox{ if } j= 0,\\
0 &  \mbox{ if } j=1, \ldots N-2,\\
0
& \mbox{ if } j=N-1. \\
\end{cases}
\end{equation}
Using the boundary conditions, we have
\begin{equation*}
\begin{split}
\alpha_{-1}^{n+1}&= \alpha_{1}^{n+1}+ \left(
h\beta+1 \right) \alpha_{0}^{n+1},\\
\alpha_{-1}^{n}&= \alpha_{1}^{n}+ \left(
h\beta + 1 \right) \alpha_{0}^{n},\\
\alpha_{N}^{n+1}&= \frac{1}{\beta h +
1} \alpha_{N-1}^{n+1},\\
\alpha_{N}^{n}&= \frac{1}{\beta h + 1} \alpha_{N-1}^{n}.
\end{split}
\end{equation*}
From the initial condition we get
$\alpha_{0}^{0}= \alpha_{N}^{0}= 0$.
Setting
$$
a= \frac{h}{6} -\frac{\tau}{h},
\quad b= \frac{2h}{3} +\frac{2 \tau}{h},
$$
and using together \eqref{eq14b}--\eqref{eq35}, we then get the
following  system of $N-1$ linear algebraic equations:\\
for $j=0$,
\begin{equation*}
\left(a(1 + h \beta)+ b +\frac{\tau \beta}{2} \right)
\alpha_{0}^{n+1} + 2a \alpha_{1}^{n+1} =  \frac{h}{6}\left(5 + h
\beta \right)\alpha_{0}^{n} + \frac{h}{3}
\alpha_{1}^{n} + \frac{2\lambda  \tau f(\alpha_{0}^{n})}{( f(\alpha_{0}^{n})
+f(\alpha_{N}^{n}))^{2}},
\end{equation*}
for $j= 1, \ldots, N-2$,
\begin{equation*}
a \alpha_{j-1}^{n+1} + b \alpha_{j}^{n+1} + a \alpha_{j+1}^{n+1}
= \frac{h}{6} \alpha_{j-1}^{n} + \frac{2h}{3}\alpha_{j}^{n}
+\frac{h}{6} \alpha_{j+1}^{n} \, ,
\end{equation*}
for $j= N-1$,
\begin{equation*}
a \alpha_{N-2}^{n+1} + \left( b+\frac{a}{1+ \beta h}\right )
\alpha_{N-1}^{n+1} = \frac{h}{6}\alpha_{N-2}^{n} + \frac{2h}{3}
\left ( 1+\frac{1}{4(1+ \beta h)} \right) \alpha_{N-1}^{n}.
\end{equation*}
Similarly,
$$
\varphi^{n+1}= \sum_{i=-1}^{N}\mu_{i}^{n+1} v_{i},
$$
where $\mu_{i}^{n+1}$ are unknown real coefficients to be
determined. The discretization of the boundary conditions with
respect to $\varphi$ looks as follows:
\begin{equation*}
\begin{split}
\mu_{-1}^{n+1} &= \mu_{1}^{n+1} + ( h \beta + 1) \mu_{0}^{n+1},\\
\mu_{-1}^{n} &= \mu_{1}^{n} + ( h \beta + 1) \mu_{0}^{n},\\
\mu_{N}^{n+1} &= \frac{1}{1+ \beta h} \mu_{N-1}^{n+1},\\
\mu_{N}^{n} &= \frac{1}{1+ \beta h} \mu_{N-1}^{n}.
\end{split}
\end{equation*}
If we set
$$
c= -\frac{h}{6} -\frac{\tau}{h},
\qquad d =- \frac{2h}{3} +\frac{2 \tau}{h},
$$
then the remaining discrete equations, approximating the
optimality system, are as follows:\\
for $j=0$,
\begin{multline*}
\left ( c(1+h \beta) + d+ \frac{\tau \beta} {2}- \frac{2\lambda
\tau \beta f'(\alpha_{0}^{n}) }{\left( f(\alpha_{N}^{n})+
f(\alpha_{0}^{n})\right)^{2}} \right) \mu_{0}^{n+1} + 2 c
\mu_{1}^{n+1}\\
= -\frac{h}{6}(5+ h \beta) \mu_{0}^{n} -\frac{h}{3}
\mu_{1}^{n} +  \tau h + \frac{2 \lambda \tau (\mu_{0}^{1}+
\mu_{1}^{N}) f(\alpha_{0}^{n})}{\left( f(\alpha_{N}^{n})+
f(\alpha_{0}^{n}) \right)^{3}},
\end{multline*}
for $j= 1, \ldots N-2$,
$$
c \mu_{j-1}^{n+1} +d \mu_{j}^{n+1}+ c \mu_{j+1}^{n+1} = \tau h -
\frac{h}{6} \mu_{j-1}^{n}- \frac{2h}{3} \mu_{j}^{n} - \frac{h}{6}
\mu_{j+1}^{n},
$$
for $j=N-1$,
$$
c \mu_{N-2}^{n+1} + \left ( d+\frac{c}{1+ \beta h} \right)\mu_{N-1}^{n+1}
=\tau h - \frac{h}{6} \mu_{N-2}^{n}- \frac{2h}{3} \left(1
+ \frac{1}{4(1+\beta h)} \right)\mu_{N-1}^{n}.
$$

Finally, we have the discretization of $\beta$ as follows:
\begin{equation}
\label{eq:beta1} \beta^{n+1} = min \left ( max \left ( m,
-\frac{u^{n+1} \varphi^{n+1}}{2}
  \right ), M \right).
\end{equation}


The numerical experiments are in agreement
with the results of \citep{MR2348482}:
we obtain stable steady-state (see Figure~\ref{fig1}).
\begin{figure}
\begin{minipage}[t]{0.45\textwidth}
\begin{center}
\includegraphics[scale=0.4]{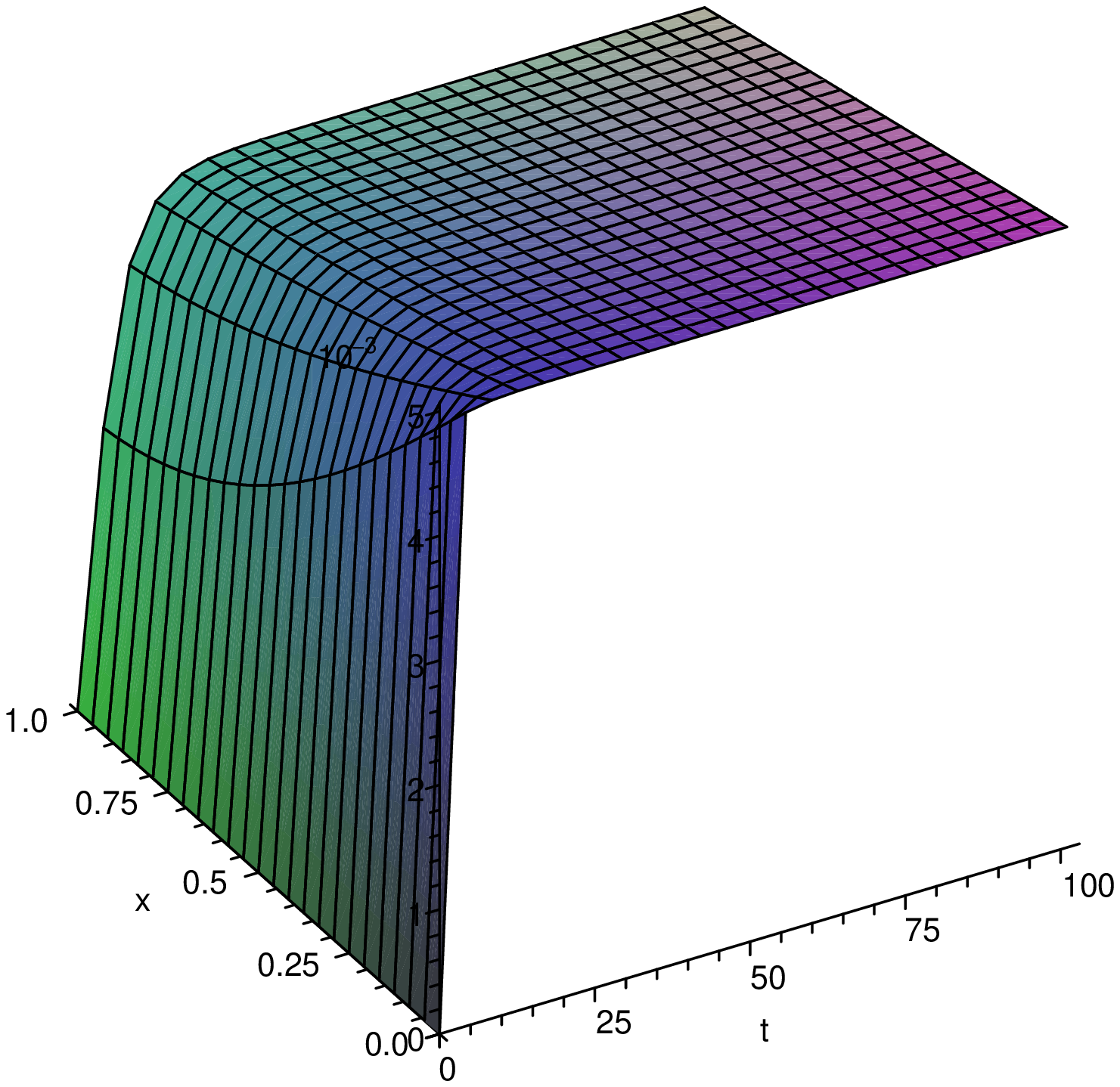}
\caption{\label{fig1} The evolution of temperature $u$}
\end{center}
\end{minipage}
\hfill
\begin{minipage}[t]{0.45\textwidth}
\begin{center}
\includegraphics[scale=0.4]{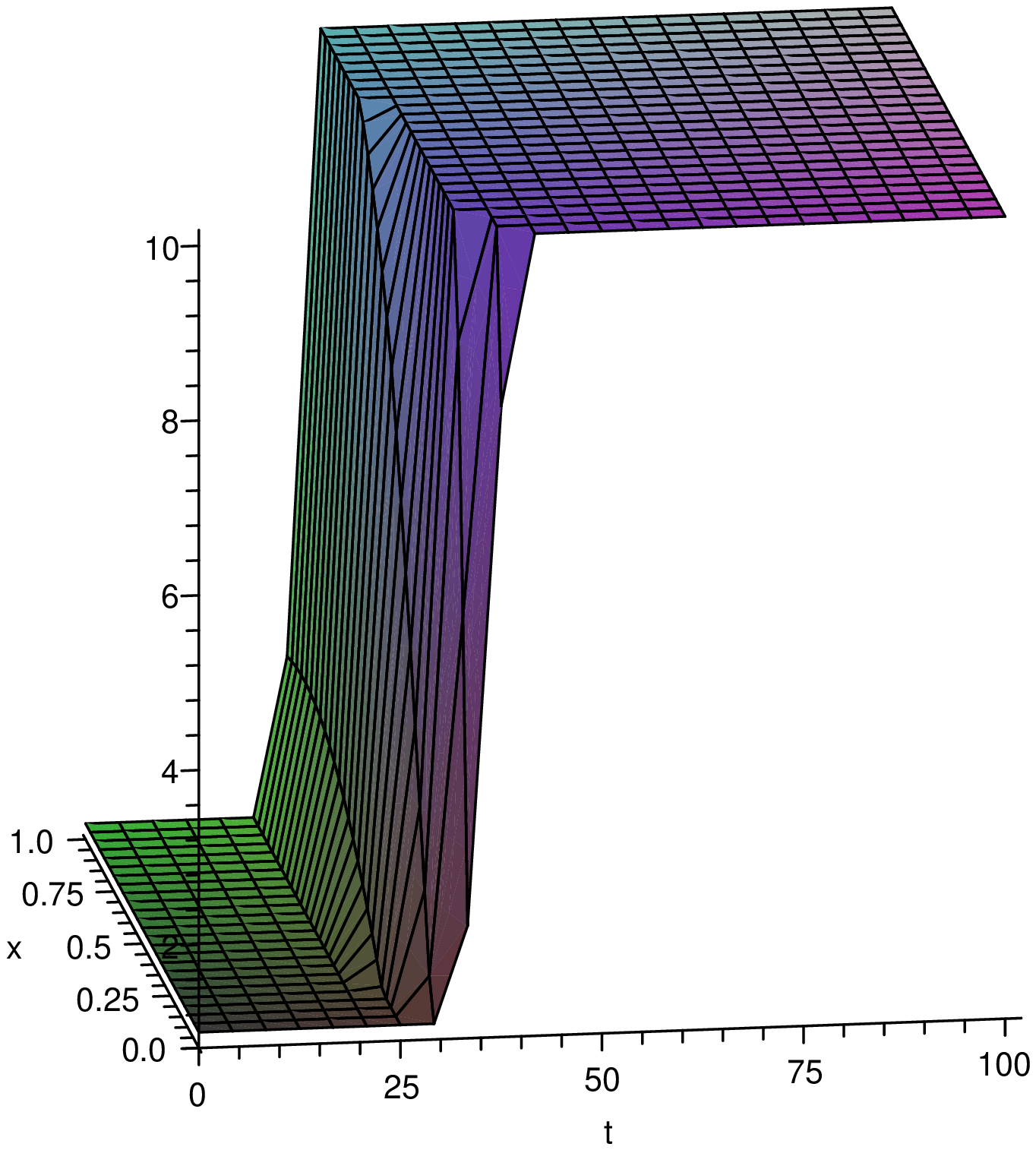}
\caption{\label{fig3} The control $\beta$}
\end{center}
\end{minipage}
\end{figure}
With an initial guess for the value of the control,
the consecutive values of $\beta$ converge to the lower bound
when time is small and to the upper bound when $t$ is big
(see Figure~\ref{fig3}).


\section*{Acknowledgements}

This work was supported by FEDER funds through
COMPETE --- Operational Programme Factors of Competitiveness
(``Programa Operacional Factores de Competitividade'')
and by Portuguese funds through the
{\it Center for Research and Development
in Mathematics and Applications} (University of Aveiro)
and the Portuguese Foundation for Science and Technology
(``FCT --- Funda\c{c}\~{a}o para a Ci\^{e}ncia e a Tecnologia''),
within project PEst-C/MAT/UI4106/2011
with COMPETE number FCOMP-01-0124-FEDER-022690.
The authors were also supported by the project
\emph{New Explorations in Control Theory Through Advanced Research} (NECTAR)
cofinanced by FCT, Portugal, and the \emph{Centre National de la Recherche
Scientifique et Technique}, Morocco.



\end{document}